\documentclass{article}
\usepackage[english]{babel}
\usepackage{amsfonts}
\usepackage{amsmath,amsthm,amssymb,latexsym}
\usepackage{color}
\usepackage{authblk}
\usepackage{graphicx}
\usepackage{tikz}
\usepackage{stmaryrd}
\usepackage{mathrsfs,upgreek}
\usepackage{eucal}
\usepackage{enumerate}
\usepackage{changes}

\newtheorem{theorem}{Theorem}[section]
\newtheorem{lemma}[theorem]{Lemma}

\newtheorem{corollary}[theorem]{Corollary}

\theoremstyle{definition}

\newcommand{\op}[1]{\textrm{\upshape #1}}

\newcommand{\join}{\vee}

\newcommand{\meet}{\wedge}

\newcommand{\la}{\langle}
\newcommand{\ra}{\rangle}
\newcommand{\alg}[1]{{\textbf{\upshape #1}}}  %
\newcommand{\vv}[1]{\mathsf {#1}}

\renewcommand{\d}{\delta}
\newcommand{\f}{\varphi}
\newcommand{\g}{\gamma}
\newcommand{\e}{\varepsilon}
\renewcommand{\th}{\theta}
\renewcommand{\o}{\omega}

\newcommand{\sse}{\subseteq}

\newcommand{\app}{\approx}
\newcommand{\VV}{{\mathbf V}}   

\newcommand{\ib}{\item[$\bullet$]}

\newcommand{\Con}[1]{\operatorname{Con}(\alg #1)}

\newcommand{\vuc}[2]{#1_1,\dots,#1_{#2}}

\newcommand{\imp}{\rightarrow}

\makeatletter
\def\square{\RIfM@\bgroup\else$\bgroup\aftergroup$\fi
  \vcenter{\hrule\hbox{\vrule\@height.6em\kern.6em\vrule}\hrule}\egroup}
\makeatother
\newcommand{\smlcirc}{\raise3pt\hbox{\textrm{\circle{3.3}}}}

\newcommand{\myfrac}[2]{\dfrac{#1}{\lower.5ex\hbox{$#2$}}}
\mathchardef\hu="0362

\renewcommand{\e}{\varepsilon}

\newcommand{\lr}{ {\slash}}
\newcommand{\rr}{ {\backslash}}

\begin{document}

\title{Why most papers on filters are really trivial (including this one)}
\author{Paolo Aglian\`{o}\\
DIISM\\
Universit\`a di Siena\\
agliano@live.com
\date{}
}
\maketitle
\begin{abstract} The aim of this note is to show that many papers on various kinds of filters (and related concepts) in (subreducts of) residuated structures
are in fact easy consequences of more general results that have been known for a long time.
\end{abstract}

This paper is born out of frustration. As most of my colleagues I am often asked to review papers submitted to journals in fuzzy logic or abstract algebraic logic; and a large majority of them
deals with some kind of particular {\em filters} on some particular structure. Of course we all know that (usually) these papers are very weak and mostly useless but they keep appearing,
cluttering the field and forcing good people (who would love to do otherwise) to read them at least once and spend some precious time in writing a rejection note (not to mention the Editors who have to deal
with this disgrace on a daily basis).  This of course is far from being news; already ten years ago a very amusing paper was published on the subject \cite{Vita2014} and the description of the phenomenon was so good that
we must (shamelessly) borrow it.

\begin{quote}\em ''We do not want to increase the amount of papers about particular, artificial types of filters. We want to illuminate the
triviality of the theory behind these papers. Proofs of presented general claims are short and clear in contrast to proofs
of particular results for concrete special types of filters which are technical and they seem like ``math exercises''. We also
want to provide a tool for reviewers who battle with dozens of papers dealing with unmotivated types of filters.'' \end{quote}

In spite of the author's intentions the situation now is worse than 10 years ago; not only the number of papers about filters has increased, introducing more and more preposterous definitions, but this
craziness has spilled over the boundary of residuated lattices, involving subreducts or other kinds of derived structures. Just a clarification; we do not mean that every paper dealing with ``interesting subsets'' of (subreducts of) residuated structures is trivial. However we believe that most of them lack a real mathematical motivation and a good chunk in the majority of them  consists of straightforward corollaries of a general theory that has been available (on respectable journals) for almost 30 years.   In conclusion an update is due; we have chosen to treat the argument
in the very general setting of universal algebra, in which a substantial theory of filters (or ideals) is already available.

We stress that in this paper we will not produce any new mathematics; our aim is rather the opposite, i.e. to show that some ``new" mathematics is not new at all.

\section{What is an ideal?}

Given an algebra $\alg A$ an ideal of $A$ is an ``interesting subset'' of the universe $A$ that may or may not be a subalgebra of $\alg A$; an example of the first kind is a normal subgroup of the group and of the
second kind is a (two-sided) ideal of a ring. Now defining what ``interesting'' means is largely a matter of taste; however there is a large consensus among the practitioners of the field that:
\begin{enumerate}
\ib an ideal must have a simple algebraic definition;
\ib  ideals must  be closed under arbitrary intersections, so that a closure operator can be defined in which the ideals are exactly the
closed sets;  this gives raise to an algebraic lattice  whose elements are exactly the ideals;
\ib ideals must convey meaningful information on the structure of the algebra.
\end{enumerate}
The three points above are all satisfied by classical ideals on lattices and of course by ideals on a set $X$. We have however to be careful here; an ideal on a set $X$ is an ideal (in the lattice sense) on the Boolean algebra of subsets of $X$. There also a significant difference between ideals on lattices and ideals on Boolean algebras; in Boolean algebras an ideal is always the $0$-class  of a suitable congruence of the algebra (really, of exactly one congruence), while this is not true in general for lattices. As a matter of fact, identifying the class of (lower bounded) lattices in which every ideal is the $0$-class of a congruence is a difficult problem which is still unsolved, up to our knowledge. Of course the same property is shared by normal subgroups of a group and (two-sided) ideals of a ring (since they are both congruence kernels).

The problem of connecting ideals of general algebras to congruence classes has been foreshadowed in \cite{Fichtner1970} but really tackled by A. Ursini in his seminal paper \cite{Ursini1972}. Later, from the late 1980's to the late 1990's, A. Ursini and the author published a long series of papers on the subject (see for instance \cite{OSV4} and the bibliography therein); the theory developed in those paper will constitute the basis of our investigation.

\section{Ideals in universal algebra}\label{ideals}

We postulated that an ideal must have a simple algebraic definition; as imprecise as this concept might be, in our context there is a natural path to follow. Given a type (a.k.a. a signature) $\sigma$  we can consider the {\bf $\sigma$-terms}  (i.e. the elements of $\alg T_\sigma(\o)$, the absolutely free countably generated algebra OF type $\sigma$; a term is denoted by $p(\vuc xn)$ to emphasize the variable involved and we will use the vector notation $\vec x$ for $\vuc xn$.
Let $\Gamma$ be a set of $\sigma$-terms; we will divide the (finite) set of variables $\vuc z{n+m}$ of each terms in two subsets $\{\vuc xn\}$ and $\{\vuc ym\}$ so that every term in $\Gamma$ can be expressed as $p(\vec x,\vec y)$ and we allow $n=0$, while $m$ must always be at least $1$. If $\alg A$ has type $\sigma$ a $\Gamma$-ideal of $\alg A$ is an $I \sse A$ such that for any $\vuc an \in A$ and $\vuc bm \in I$, $p(\vec a,\vec b) \in I$.
The following is a simple exercise.

\begin{lemma} Let $\sigma$ be any type,  $\Gamma$ a set of $\sigma$-terms and $\alg A$ an algebra of type $\sigma$. Then
\begin{enumerate}
 \item the $\Gamma$-ideals of $\alg A$ are closed under arbitrary intersections;
 \item the $\Gamma$-ideal generated by $X \sse A$, i.e. the intersection of all the $\Gamma$-ideals containing $X$, is
$$
\op{Id}_\alg A^\Gamma(X) = \{p(\vec a,\vec b): \vec a \in A, \vec b \in X\};
$$
\item  the $\Gamma$-ideals of $\alg A$ form an algebraic lattice $\op{Id}^\Gamma(\alg A)$.
\end{enumerate}
\end{lemma}

At this level of generality we cannot say much more; if the type however contains a constant we can get a more focused definition. Let $\vv V$ be a variety whose type contains a constant which will denote by $0$; a
{\bf $\vv V,0$-ideal term in $\vuc ym$} is a term $p(\vec x,\vec y)$ such that
$$
\vv V \vDash p(\vec x,0,\dots,0) \app 0.
$$
Let $ID_{\vv V,0}$ be the set of all $\vv V,0$-ideal terms in $\vv V$; a {\bf $\vv V,0$-ideal}  $I$ of $\alg A \in \vv V$ is a $ID_{\vv V,0}$-ideal of $\alg A$.  If $\vv V = \VV(\alg A)$ we will simply say that $F$ is a {\bf $0$-ideal} of $\alg A$. As before the set $\op{Id}_{\vv V,0}(\alg A)$ of $\vv V,0$-ideals of $\alg A$ and the set $\op{Id}_0(\alg A)$ of $0$-ideals of $\alg A$ are algebraic lattices and $\op{Id}_0(\alg A) \sse \op{Id}_{\vv V,0}(\alg A)$ (and the inclusion may be strict). It is also evident that for any $\th \in \Con A$, $0/\th$ is a $\vv V,0$-ideal of $\alg A$: if $p(\vec x,\vec y) \in ID_{\vv V,0}$, $\vec a \in A$ and $\vec b \in 0/\th$ then
$$
p(\vec a,\vec b)\mathrel{\th} p(\vec a,\vec 1) =0.
$$
We say that $\vv V$ has {\bf normal $\vv V,0$-ideals} if for all $\alg A \in \vv V$ for all $I \in \op{Id}_{\vv V,0}(\alg A)$ there is a $\th \in \Con A$ with $I = 0/\th$. If $\vv V$ has normal $\vv V,0$-ideals then
of course $\op{Id}_0(\alg A) = \op{Id}_{\vv V,0}(\alg A)=\{0/\th: \th \in \Con A\}$ so we can simply talk about $0$-ideals of $\alg A$ without specifying the variety. Observe that the variety of pointed (by $0$) sets has normal
$0$-ideals, so we can hardly expext any nice structural theorem for varieties with $0$-normal ideals. However something can be said and the interested reader can consult \cite{OSV2} for more information.

We say that $\vv V$ is {\bf $0$-subtractive} or simply {\bf subtractive} \cite{OSV1} if there is a binary term $s(x,y)$ in the type of $\vv V$ such that
$$
\vv V\vDash s(x,x) \app 0\qquad\vv V\vDash s(x,0) \app x.
$$

\begin{theorem} For a variety $\vv V$ the following are equivalent:
\begin{enumerate}
\item $\vv V$ is subtractive;
\item for every $\alg A \in \vv V$ and $\th,\f \in \Con A$, $0/\th \join \f = 0/\th \circ \f$ (hence the congruences permute at $0$);
\item for every $\alg A \in \vv V$ the mapping $\th \longrightarrow 0/\th$ is complete and onto lattice homomorphism from $\Con A$ to $\op{Id}_{\vv V,0}(\alg A)$.
\end{enumerate}
\end{theorem}
For the proof and for even more equivalent the reader can look at  Theorem 2.4 in \cite{AglianoUrsini1992}. As an immediate consequence we have

\begin{corollary}\label{cormain} Let $\vv V$ be a subtractive variety; then
\begin{enumerate}
\item for every $\alg A \in \vv V$, $\op{Id}_{\vv V,0}(\alg A)$ is a modular lattice;
\item $\vv V$ has normal ideals.
\end{enumerate}
\end{corollary}

Let $\vv V$ be subtractive, $\alg A \in \vv V$ and $I \in \op{Id}_0(\alg A)$; let
$$
I^d = \bigwedge\{\th \in \Con A: 0/\th=I\}\qquad I^\e = \bigvee\{\th \in \Con A: 0/\th=I\}.
$$
Then the interval $[I^\d,I^\e]$ in $\Con A$  consists of all $\th \in \Con A$ such that $I=0/\th$. The properties of the mapping $I \longmapsto I^\e$ have been investigated at length in \cite{OSV3} and it turns out (perhaps non surprisingly) that they are connected to abstract algebraic logic.  However in this paper we do not need to deal with such intricacies; we simply need to consider the case in which the mapping is good enough to allow a   connection between ideals and congruences. A subtractive variety $\vv V$ is  {\bf (finitely) congruential} \cite{OSV3} if there are binary terms $\vuc dn$ of $\vv V$ such that $\vv V \vDash d_i(x,x) \app 0$ for $i=1,\dots,n$ and
for all $\alg A \in \vv V$ and $I \in \op{Id}_0(\alg A)$
$$
I^\e = \{(a,b): d_i(a,b) \in I: i=1,\dots,n\}.
$$

\begin{theorem}\label{maincon} \cite{OSV3} For a subtractive variety $\vv V$ the following are equivalent:
\begin{enumerate}
\item $\vv V$ is finitely congruential witness $\vuc dn$;
\item the mapping $()^\e$ is continuous, i.e. for $\alg A \in \vv V$ and every family $(I_\g)_{\g \in \Gamma}$ of $0$-ideals of $\alg A$
$$
(\bigcup_{\g \in \Gamma} I_\g)^\e = \bigcup_{\g \in \Gamma} I^\e.
$$
\item there exist binary terms $\vuc dn$, an $n+3$-ary term $q$ of $\vv V$ and for each basic operation $f$ of arity $k$ and $i=1,\dots,n$  a (2+n)k-ary term $r_{i,f}$ such that
\begin{align*}
&\vv V \vDash d_(x,x) \app 0 \qquad i=1,\dots,n \\
&\vv V \vDash q(x,y,0,0,\dots,0) \app 0\\
&\vv V \vDash q(x,y,y,d_1(x,y),\dots,d_n(x,y)) \app x\\
&\vv V \vDash r_{i,f}(\vec x,\vec y, 0,\dots,0) \app 0\\
&\vv V \vDash r_{i,f}(\vec x,\vec y,d_1(x_1,y_1),\dots,d_1(x_k,y_k),\dots,d_n(x_1,y_1),\dots,d_n(x_k,y_k)).
\end{align*}
\end{enumerate}
\end{theorem}

As a particular instance of being congruential we can to consider the case in which the interval $[I^\d,I^\e]$ degenerates to a point, i.e. $0/\th = 0/\f$ implies $\th=\f$. In this case it can be shown that the mapping
$I \longmapsto I^\e$ is in fact an isomorphism and that $I^\e$ is the unique $\th \in \Con A$ with $0/\th = I$.  In this case the variety $\vv V$ is called {\bf $0$-regular} and we have

\begin{corollary}\cite{GummUrsini1984}  For a pointed (at $0$) variety $\vv V$ the following are equivalent:
\begin{enumerate}
\item $\vv V$ is subtractive and $0$-regular;
\item $\vv V$ is subtractive and there is an $n \in \mathbb N$ and binary terms $\vuc dn$ such that
\begin{align*}
&\vv V\vDash d_i(x,x) \app 0\qquad\text{$i=1,\dots,n$}\\
&\vv V \vDash \{d_i(x,y) \app 0: i=1,\dots,n\} \Rightarrow  x \app y.
\end{align*}
\end{enumerate}
\end{corollary}

A pointed variety which is subtractive and $0$-regular is called {\bf ideal determined} \cite{GummUrsini1984}; if $\vv V$ is such a variety then for any algebra $\alg A \in \vv V$, $\op{Id}_0(\alg A) \cong \Con A$ and
hence $\vv V$ is congruence modular by Corollary \ref{cormain}.  Clearly groups, rings, vector spaces, boolean algebras and many other classical algebras are ideal determined and so are residuated lattices, $\mathsf{FL}$-algebras and many of their subreducts. They are also congruence permutable in most cases; however for instance implication algebras are not congruence permutable \cite{Mitschke1971} but it is easy check that they are ideal determined.
There are also non ideal determined varieties to which Theorem \ref{maincon} applies, such as the variety of pseudocomplemented semilattices (\cite{OSV3}, Example 4.4).

Because of Theorem \ref{maincon} a finitely congruential subtractive variety $\vv V$ has two features that are the prototype of many papers on filters on residuated structures.
Let $T$ be a set of terms of $\vv V$; a $0$-ideal $I$ of $\alg A \in \vv V$ is {\bf $T$-special} if for all $\vec a \in A$, $t(\vec a) \in I$ for all $t \in T$.  It is obvious that the property of being $T$-special is upward hereditary on ideals; it is also obvious that the $T$-special ideals form an algebraic lattice.  Moreover let $\vv V_T$ be the subvariety of $\vv V$ axiomatized by the equations $t(\vec x) \app 0$ for $t \in T$.

\begin{lemma}\label{special} Let $\vv V$ be a finitely congruential subtractive variety, witness $\vuc dn$. Then for any  set $T$ of terms of $\vv V$ and $\alg A \in \vv V$, $\alg A/I^\e \in \vv V_T$ if and only if
$I$ is $T$-special.
\end{lemma}
\begin{proof} Suppose that $I$ is $T$ special; then for all $\vec a \in A$, $t(\vec a) \in I$ for all $t \in T$. As each $d_i(x,y)$ is an ideal term, we get that $d_i(t(a),0) \in I$ for $i=1,\dots,n$ and therefore
$(t(a),0) \in I^\e$ for all $t \in T$. This implies that $\alg A/I^\e \in \vv V_T$ as wished.

Conversely, suppose that $\alg A/I^e \in \vv T$; then for all $\vec a \in A$, $(t(\vec a),0) \in I^\e$ for all $t \in T$.  This implies that $d_1(t(\vec a),0),\dots,d_n(t(\vec a),0) \in I$ for all $t \in T$; by Theorem \ref{maincon}(3)
$$
t(\vec a) = q(t(\vec a),0,0,d_1(t(\vec a),0),\dots,d_n(t(\vec a),0)) \in I
$$
for al $t \in T$ and thus $I$ is $T$-special.
\end{proof}

So in a finitely congruential variety we have a potentially unlimited supply of $T$-special filters; if the variety is also ideal determined, then even more is true. Let $\vv V'$ be a subvariety of $\vv V$ and let $J$ be a set of equations axiomatizing $\vv V'$ relative to $\vv V$. If we set
$$
T= \{d_i(p,q): i =1,\dots,n\  p\app q \in J\}
$$
then $\vv V\vDash t(\vec x) \app 0$ for all $t \in T$ if and only if $\vv V \vDash p \app q$ for all $p \app q \in J$.  It follows that $\vv V'= \vv V_T$ and so any subvariety of $\vv V$ can be taken as the base for defining son $T$-special ideals.

The second consequence is the following; let $\vv V ^+$ be a pointed variety such a class of subreducts $\vv V$ of $\vv V^+$ happens to be a finitely congruential subtractive variety. Certainly $\vv V^+$ is subtractive as well;
if for any ``new" operation $f$ in the type of $\vv V^+$ we can find terms $r_{i,f}$ satisfying (2) of Theorem \ref{maincon}, then $\vv V^+$ is finitely congruential as well and the ideals in $\vv V^+$ are exactly the $\vv V$-ideals that are closed under all the $r_{i,f}$, where $f$ is a new operation. A particularly easy case is the one in which the new operation is itself a {\em pure} ideal term, i.e. $f(0,\dots,0) \app 0$.

\section{Variations on $\mathsf{FL}$-algebras}

A {\bf residuated lattice} is an algebra  $\alg A = \la A,\join,\meet,\cdot,\lr,\rr, 1\ra$ where
\begin{enumerate}
\item $\la A, \join, \meet\ra $ is a lattice;
\item $\la A, \cdot,1\ra$ is a monoid;
\item $\lr$ and $\rr$ are the left and right residua w.r.t. $\cdot$, i.e. $x \cdot y \leq z$ iff $y \leq x \rr z$ iff $x \leq z \lr y$.
\end{enumerate}
Residuated lattices form a variety $\mathsf{RL}$ and an axiomatization, together with the many equations holding in these very rich structures, can be found in \cite{BlountTsinakis2003}. A residuated lattice is {\bf integral} if satisfies $x \le 1$ and {\bf commutative} if the monoidal operation is commutative. An $\mathsf{FL}$-algebra is an algebra $\alg A = \la A,\join,\meet,\cdot,\lr,\rr,0, 1\ra$ where $\alg A = \la A,\join,\meet,\cdot,\lr,\rr, 1\ra$ is a residuated lattice and $0$ is a constant. An $\mathsf{FL}$-algebra is
\begin{enumerate}
\ib a $\mathsf{FL_{w}}$-algebra, if it is integral and satisfies $0\le x$,
\ib a $\mathsf{FL_e}$-algebra, if it is commutative,
\ib a $\mathsf{FL}_{ew}$-algebra, if it is both an $\mathsf{FL}_w$ algebra and an $\mathsf{FL}_e$-algebra.
\end{enumerate}

Residuated lattices are clearly ideal determined, hence so are $\mathsf{FL}$-algebras. Let $\alg A$ be a residuated lattice and let $A^+= \{a \in A: a \ge 1\}$; a  {\bf filter} of $\alg A$ is a subset $F \sse A$ such that
\begin{enumerate}
\item $A^+ \sse \in F$;
\item $a, a \lr b \in F$ implies $b \in $;
\item $a,b \in f$ implies $a \meet b \in F$.
\end{enumerate}
A filter $F$ is {\bf normal} if
\begin{enumerate}
\item[4] $a \in F$ and $b \in A$ implies $b \rr ab, ba\lr b \in F$.
\end{enumerate}
Clearly in any commutative residuated lattices (or $\mathsf{FL}_e$-algebra) every filter is a normal filter.
It is well-known that there is a one-to-one correspondence (which is in fact a lattice isomorphism) between the normal filters and the congruences of $\alg A$ given by the mutually inverse maps
$$
\th \longmapsto A^+/\th \qquad\qquad F \longmapsto \th_F= \{(a,b):  a \lr b,b\lr a \in F\}.
$$
Now if $\alg A$ is integral, then $A^+ = \{1\}$ and hence the normal filters are just the $1$-ideals; if $\alg A$ is not integral then the $1$-ideals are not filters but rather  the {\bf convex normal subalgebras} of $\alg A$.
However there is a straightforward way to connect convex normal subalgebras and normal filters in such a way that the results of Section \ref{ideals} can be transferred easily. We spare the technical details mainly because all the examples of ``papers of filters'' that we will introduce deal in fact with normal filters in integral residuated lattices (or $\mathsf{FL}_w$-algebras). Indeed in the majority of cases commutativity of the monoid operation is also present. We remind that subvarieties of $\mathsf{FL}_{ew}$ have been studied extensively in the literature; examples are the variety of $\mathsf{MV}$-algebras \cite{Panti1999}, $\mathsf{BL}$-algebras \cite{AglianoMontagna2003} and $\mathsf{MTL}$-algebras \cite{EstevaGodo2001}.

From our considerations it follows that any  theory of ``special" normal filters in integral residuated lattices (or $\mathsf{FL}_w$-algebras)  can be regarded as a special case of the theory of special ideals in pointed varieties. So if one wants to create to create a (bad) paper on filters of say, $\mathsf{BL}$ or $\mathsf{FL}_{ew}$-algebras all he  has to do is to find some term and come up with a fancy name for the filters that are $T$-special w.r.t. to those equations. Then he can  prove a bunch of results that are corollaries of the results in Section \ref{ideals}, but of course, without using the general theory, the proofs very often consist in long (and pointless) calculations.  These results have been classified in \cite{VitaCintula2011}, which is the ``serious'' counterpart of the more {\em tongue-in-cheesk}  \cite{Vita2014}; the list of these amenities is substantial and we believe there is no need to produce some more.

However we would like to point out that the same trick has been applied to varieties which consist of ideal determined subreducts of varieties of
$\mathsf{FL}$-algebras: the trick here is to take away some operation from the type of $\mathsf{FL}$-algebra,  still keeping ideal determinacy.
It is easily seen that if $\vv V$ is any variety of $\mathsf{FL}_w$-algebras, then any class of subreducts of $\vv V$
that contains  $\{\lr,\rr,\meet,1\}$ in its type is an ideal determined variety.

Again let's deal with the integral case to make things simpler, i.e. $\mathsf{FL_{w}}$-algebras.   {\bf Pseudo-$\mathsf{BL}$-algebras} are just the non commutative version of $\mathsf{BL}$-algebras; {\bf pseudohoops} have been introduced and defined via equations in \cite{GeorgescuLeusteanProtasa} but it is clear that they are just the $0$-less subreducts of pseudo-$\mathsf{BL}$-algebras and  of course pseudohoops and pseudo-$\mathsf{BL}$-algebras are ideal determined.   So for instance the entire \cite{AlaviBorzoeiKologani} is a more or less trivial consequence of \cite{KondoDudek}, which is a particular instance of \cite{VitaCintula2011}, which is a consequence of the general ideal theory in Section \ref{ideals}.
Now the reader can easily verify for instance that  the more recent  \cite{BorzoeiNamdarKologani}, \cite{Ciungu} and \cite{MotamedMoghaderi2019}  are just more of the same.

\section{Adding operations}

Another nice trick is to add new operations to $\mathsf{FL}$-algebras, but the axioms for those new operations are carefully chosen in such a way that the
{\em new} filters are just the {\em old} filters that are closed under those additional operations. Of course there is a very general way to do that (see \cite{Agliano1996b}); here we will briefly show how to do it for commutative and integral residuated lattices (and hence for $\mathsf{FL}_{ew}$-algebras).

Let $\alg A$ be a commutative and integral residuated lattice; a unary operation $h$ on $\alg A$ is {\bf normal} if for all $a,b \in A$
\begin{enumerate}
\ib $h(1)=1$;
\ib $h(a \imp b) \le h(a) \imp h(b)$.
\end{enumerate}
Observe that by the second point above any normal operation is increasing: $a \le b$ implies $f(a) \le f(b)$.  An $n$-ary operation $f$ on $A$ is {\bf normal} if for all $i\le n$ and for all $a_1,\dots,a_{i-1},a_{i+1},\dots,a_n \in A$
$$
f_i(x) := f(a_1,\dots,a_{i-1},x,a_{i+1},\dots,a_n)
$$
is normal. A {\bf commutative integral residuated lattice with normal operators} is an algebra $\alg A =\la A,\join,\meet,\imp,\cdot,1, f_\lambda\ra_{\lambda in \Lambda}$ such that $\la A,\join,\meet,\imp,\cdot,1\ra$ is a commutative and integral residuated lattice and $f_\lambda$ is normal for any $\lambda \in \Lambda$.

\begin{theorem} (see \cite{Agliano1996b}, Proposition 3.5)  If $\alg A =\la A,\join,\meet,\imp,\cdot,1, f_\lambda\ra_{\lambda in \Lambda}$ is a commutative integral residuated lattice with normal operators, then the $1$-ideals (which we may call the filters) of $\alg A$ are exactly the filters of the commutative residuated lattice reduct that are closed under $f_\lambda$ for any $\lambda \in \Lambda$.
\end{theorem}

If $\alg A$ is as above, then we denote by $P(\alg A)$ the set of all unary polynomials of $\alg A$ involving only the normal operators.
If $X \sse A$, let $\op{Fil}_\alg A(X)$ be the filter generated by $X$ in $\alg A$.

\begin{lemma} (see Proposition 3.6 in \cite{Agliano1996b}) Let $\alg A$ be a commutative integral residuated lattice with normal
operators.
\begin{enumerate}
\item If $X \sse A$, then  $a \in \op{Fil}_\alg A (X)$ if and only if there are $\vuc bn
\in X$
and $\vuc pn \in \op{P}(\alg A)$ such that
$$
(p_1(b_1)\meet 1)\dots(p_n(b_n) \meet 1) \le a.
$$
\item  If $F,G$ are filters of $\alg A$,
$$
F \join G =\{c : ab \le c,\ \text{for some}\ a \in F,\ b \in G\}.
$$
\end{enumerate}
\end{lemma}

Now this very general result, when applied to specific cases, can be made as complex as we want. Take for instance \cite{Paad2021}; there the author
defines a {\bf tense $\mathsf{BL}$-algebra} has a $\mathsf{BL}$-algebra with two {\em tense} unary operators. As this operators are clearly normal, the
general theory applies and many of the results in the paper are simply a straightforward consequence. I daresay that the general formulation is even clearer than the particular one (which is notationally heavy), but this is just a matter of opinion.

Another popular topic is  adding modal operators to residuated structures. There nothing wrong in that of course; for instance in \cite{Castanoetal}  the authors proposed a modal calculus with the two classical modalities $\Box$ and $\Diamond$  based on Hajek's Basic Logic. Its equivalent algebraic semantics consists of structures  $\la \alg A,\Box,\Diamond\ra$ where $\alg A$ is a $\mathsf{BL}$-algebra and $\Box, \Diamond$ are two unary operators satisfying certain equations. Filters are defined as filters (i.e. $1$-ideals) of $\alg A$ closed under $\Box$ which is a normal operator. The axiom chosen basically imply that the congruences of $\la \alg A,\Box,\Diamond\ra$  coincide with the congruences of its reduct $\la \alg A,\Box\ra$ so the filters of $\la \alg A,\Box,\Diamond\ra$ coincide with $1$-ideals and the general theory applies. Of course the same trick can be applied (to a certain extent) to every variety of $\mathsf{FL}_{ew}$-algebras; as usual one has simply to add enough axioms in such a way that the ``new'' filters are $1$-ideals (see for instance \cite{WangHeShe2019}).

\section{Changing the constant}

$\mathsf{FL}$-algebras have two constants, $0$ and $1$ and the $1$-ideals are the filters; of course one might wonder what happens if we consider the $0$-ideals. Since $\mathsf{FL}$-algebras are highly non symmetrical (unlike Boolean algebras) one might be led to believe that the situation may be different. And indeed it is: easy examples show that there are varieties of $\mathsf{FL}$-algebras that do not have normal $0$-ideals.

However there some cases in which we can apply the general theory to the $0$-ideals. In fact the variety of $\mathsf{FL}_{ew}$-algebras happens to be
$0$-subtractive witness the term $s(x,y) := (y \imp 0)x$; so in any variety of $\mathsf{FL}_{ew}$-algebras the $0$-ideals coincide with the $0$-classes of congruences. Varieties of $\mathsf{FL}_{ew}$-algebras are not $0$-ideal determined but:

\begin{lemma} Let $\alg A$ be a $\mathsf{FL_{ew}}$-algebras and let $I$ be a $0$-ideals of $\alg A$; then
$$
I^\e = \{(a,b):  (a \imp 0)\cdot b, (b \imp 0) \cdot a \in I\}.
$$
Hence any  variety of $\mathsf{FL}_{ew}$-algebras is finitely congruential w.r.t. $0$.
\end{lemma}

So the general theory of special $T$-ideals can be applied in this case as well and (for instance) most of the results in \cite{BusneagPiciuDina2021}, \cite{Paad2017}, \cite{Paad2018} and \cite{YangSin2017}    follow from that.

\section{Tweaking the operations}

Fantasy has its limits and eventually even the most preposterous way of defining filters and ideals  over residuated lattices runs out of steam. So one can start tweaking a little bit the operations to get different structures an more complicated algebraic proofs of exactly the same results.

The first thing is to get rid of associativity of multiplication and study  {\em residuated lattice ordered groupoids}; these are very interesting structures in they own right but the filter and the ideal theories do not present much novelty. For instance in \cite{Botur2011}  {\bf non associative  {bounded} residuated lattices} are introduced: in this case the binary operation is inspired by a non associative continuous $t$-norm on $[0,1]$ and
therefore continuity forces $1$ to be the groupoid identity.
Then {\bf non associative $\mathsf{BL}$-algebras} are defined as prelinear and divisible non associative residuated lattices and
it turns out (not surprisingly) that the variety of non associative $\mathsf{BL}$-algebras is  generated as a quasivariety by all the non associative $\mathsf{BL}$-algebras induced by non associative continuos $t$-norms.
Clearly non associative residuated lattices are $1$-ideal determined and finitely congruential at $0$ so the general theory applies again.

At first I could not find any paper on special filters of such structures and I must confess I was slightly disappointed. However one should never underestimate creativity; in \cite{PaivaRivieccio} the authors investigated $t$-norms on $[0,1]$ that are not only non associative, but also non unital, in the sense
that, if $\*$ is such a $t$-norm, then $x \le x \* 1$ and the inequality may be strict. The introduction of these structures, called {\bf inflationary general residuated lattices}, is well motivated and they turn out to be
more interesting in that they are NOT ideal-determined of subtractive in general. Of course some of the general theory of ideals can be recovered in that setting as well, but it is no straightforward business and there is no need to explain it here.  We simply notice that a non associative bounded residuated lattice  is simply an inflationary general residuated lattice in which $x \* 1 \le 1$.
Now in \cite{LiangZhang2022a} the authors introduced a non commutative version of inflationary general residuated lattices with all the usual results; this is pointless enough but the authors cannot resist the temptation to study special filters \cite{LiangZhang2022b}. Of course they run into trouble, since inflationary general residuated lattices do not have a straightforward theory of filters (or ideals): their solution is to go back to non associative residuated lattices but in a covert way. Look at Theorem 1 in \cite{LiangZhang2022b}; since the statement  ``$1$ is the groupoid identity'' appears in both point (1) (implicitly) and point (2) (explicitly), this is really a statement about non associative residuated lattices. And it  really says that if one quotients out a non associative residuated lattice by a $T$-special ideal, where $T$ is any set of equations axiomatizing non associative $\mathsf{BL}$-algebras modulo non associative residuated lattices, then the result is a non associative $\mathsf{BL}$-algebra;  and, of course, this is a consequence of the general theory of $T$-special ideals. Now that the king is naked the reader can go through \cite{LiangZhang2022b} and have fun in discovering similar instances of this phenomenon.

Another generalization that is slightly different but in the same spirit is to consider a  {\em bounded q-lattice ordered residuated q-monoid} as a basis for a residuated structure. In this case we give some details that are helpful in understanding the context;  a {\bf q-lattice} is  an algebra $\la A, \join,\meet\ra$ such that
$\join$ and $\meet$ are commutative and associative and for all $a,b \in A$
\begin{enumerate}
\item $a \join (b \meet a) = a \join a = a \meet a = a \meet (b \join a)$;
\item $a \join b = a \join (b \join b)$, $a \meet b = a \meet (b \meet b)$.
\end{enumerate}
Clearly the relation $a \le b$ if $a \join a = a \join b$ is a quasiordering.
A {\bf q-monoid} is an algebra $\la A, \cdot,1\ra$ such that $\cdot$ is associative and for all $a,b \in A$
\begin{enumerate}
\item $a \cdot 1 = 1 \cdot a$;
\item $a \cdot b \cdot 1 = a \cdot b$;
\item $1 \cdot 1 = 1$.
\end{enumerate}
A {\bf quasi-$\mathsf{FL}_{w}$-algebra}  is a structure $\la \alg A,\join,\meet,\cdot,\lr,\rr,0,1\ra$ where
\begin{enumerate}
\item $\la A, \join,\meet\ra$ is a q-lattice and $0,1$ are the bottom and the top in the quasi ordering;
\item $(\lr,\cdot)$ and $(\rr,\cdot)$ are left and right residuated pairs w.r.t to $\le$;
\item $0 \meet 0 =0$ and for all $a \in A$, $a \cdot 1 = a \meet a$;
\item for all $a,b \in A$, $(b \lr a)\cdot 1 = b \lr a$ and $(a \rr b)\cdot 1 = a \rr b$.
\end{enumerate}

If $\alg A$ is a quasi-$\mathsf{FL}_w$-algebra and element  $a \in A$ is {\bf regular} if $a \cdot 1 = a$.
Let $R_\alg A$ be the set of regular elements of $\alg A$; then it is an easy exercise to check that
\begin{enumerate}
\ib $R_\alg A$ is the universe of a subalgebra $\alg R_\alg A$ of $\alg A$ that is an $\mathsf{FL}_w$-algebra;
\ib $R_\alg A = \{a \cdot 1: a \in A\}$.
\end{enumerate}
A congruence $\th \in \Con A$ is {\bf regular} if $(a\cdot 1,b \cdot 1) \in \th$ implies $(a,b) \in \th$. The proofs of the following two theorems
are straightforward:

\begin{theorem} Let $\alg A$ be a quasi-$\mathsf{FL}_w$-algebra and $\th \in \Con A$; then the following are equivalent:
\begin{enumerate}
\item $\th$ is regular;
\item $\alg A/\th$ is an $\mathsf{FL}_w$-algebra;
\item $1/\th \cap R_\alg A$ is a normal filter of $\alg R_\alg A$;
\item $1/\th = \uparrow G$ for some normal filter $G$ of $\alg R_\alg A$;
\item $\th = \{(a,b): a \lr b, b \lr a \in 1/\th\}$.
\end{enumerate}
\end{theorem}

A {\bf normal filter} of a quasi-$\mathsf{FL}_w$-algebra $\alg A$ is a subset $F \sse A$ such that
\begin{enumerate}
\ib $1 \in F$;
\ib $a \in F$ and $ a \le b$ implies $b \in F$;
\ib $a\rr b \in F$ if and only if $b\lr a \in F$.
\end{enumerate}

\begin{theorem} Let $\alg  A$ be a quasi-$\mathsf{FL}$-algebra; then
\begin{enumerate}
\item  $F \sse A$ is a normal filter if and only
if it is  equal to $1/\th$ for some regular $\th \in \Con A$;
\item $\th$ is a regular congruence of $\alg A$ if and only if
$$
\th = \{(a,b): a \lr b, b\lr a \in F\}
$$
for some normal filter $F$;
\item the regular congruences of $\alg A$ form an algebraic lattice $\op{RCon}(\alg A)$ and
the normal filters form an algebraic lattice $\op{NFil}(\alg A)$; they are isomorphic via the mapping
$$
\th \longmapsto 1/\th \qquad  F \in \th_F= \{(a,b): a \lr b, b\lr a \in F\};
$$
\item $\op{RCon}(\alg A)$ is a complete sublattice of $\Con A$;
\item  $\op{NFil}(\alg A) \cong \op{RCon}(\alg A) \cong \op{Con}(\alg R_\alg A) \cong \op{NFil}(\alg R_\alg A)$.
\end{enumerate}
\end{theorem}

It follows that the theory of normal filters in a quasi-$\mathsf{FL}_w$-algebra $\alg A$ is equivalent to the theory of normal filters (i.e. the $1$-ideals)
of its associated $\mathsf{FL}_w$-algebra $\alg R_\alg A$. At this point one can add suitable axioms to quasi-$\mathsf{FL}_w$-algebras to get subvarieties that are the quasi-replica of subvarieties of $\mathsf{FL}_w$. And for each of those the theory of normal filters is equivalent to the theory of $1$-ideals of
the corresponding subvariety of $\mathsf{FL}_w$. Examples of this are  {\bf quasi-pseudo-$\mathsf{BL}$-algebras} \cite{ChenXu2023}, {\bf quasi-pseudo-$\mathsf{MV}$-algebras} \cite{ChenDudek2018} and many others.
We also stress that we have not exhausted all the possible variations; for instance the reader can have fun in dissecting \cite{Oneretal}.

\section{Conclusions}

First we want to emphasize that there are at least two topics that we have not touched. The first one deals with algebras whose type contains a binary operation that resembles the implication and a constant $1$; the most famous (and serious) examples of algebras of this kind are $\mathsf{BCK}$ and $\mathsf{BCI}$-algebras.
Of course the poverty of the language allows the construction of many different algebraic structures with many non equivalent definitions of ``filter'' and the vast majority of them is totally pointless and utterly uninteresting. And  those that might be of some interest are those for which the filter theory corresponds to the the theory of $1$-ideals; the reader may want to look at \cite{OSV3}, Example 4.5 to understand what we mean.

The second topic we have not considered is the introduction of the so-called {\em fuzzy filters} on residuated structure; we suspect that a lot can be said in that direction as well but we did not have the stomach for it.

Finally let us give more explanations on the reasons why we have embarked in this enterprize. the main reason of course is that we believe that this is bad mathematics and should be avoided. But there are also more practical reasons, as  this way of doing mathematics gives a bad reputation to
residuated structures and, in the end, to ``contemporary'' algebraic logic. This is not an exaggeration: my original field is universal algebra and I have listened to many of my colleagues joking about it\dots  Of course one might object that most of these papers end up in fourth rate journal of worse so no harm is done; but one should not forget that those journals are indexed in Scopus for instance, so they give metrics that are commonly accepted when evaluating a researcher. How this impacts on the credibility of our field is anybody's guess (and my guess should be, at this point, clear).
In conclusion I believe that as a group we have the responsibility to police our field more effectively than  we are doing now. Other fields have found ways of doing that; maybe it is time to think seriously about it.

\providecommand{\bysame}{\leavevmode\hbox to3em{\hrulefill}\thinspace}
\providecommand{\MR}{\relax\ifhmode\unskip\space\fi MR }
\providecommand{\MRhref}[2]{%
  \href{http://www.ams.org/mathscinet-getitem?mr=#1}{#2}
}
\providecommand{\href}[2]{#2}


\begin{thebibliography}{10}

\bibitem{Agliano1996b}
P.~Aglian\`o, \emph{Ternary deduction terms in residuated structures}, Acta
  Sci. Math. (Szeged) \textbf{64} (1998), 397--429.

\bibitem{AglianoMontagna2003}
P.~Aglian\`o and F.~Montagna, \emph{Varieties of {BL}-algebras {I}: general
  properties}, J. Pure Appl. Algebra \textbf{181} (2003), 105--129.

\bibitem{AglianoUrsini1992}
P.~Aglian\`o and A.~Ursini, \emph{Ideals and other generalizations of
  congruence classes}, J. Aust. Math. Soc. \textbf{53} (1992), 103--115.

\bibitem{OSV2}
\bysame, \emph{On subtractive varieties {II}: General properties}, Algebra
  Universalis \textbf{36} (1996), 222--259.

\bibitem{OSV3}
\bysame, \emph{On subtractive varieties {III}: From ideals to congruences},
  Algebra Universalis \textbf{37} (1997), 296--333.

\bibitem{OSV4}
\bysame, \emph{On subtractive varieties {IV}: Definability of principal
  ideals}, Algebra Universalis \textbf{38} (1997), 355--389.

\bibitem{AlaviBorzoeiKologani}
S.Z. Alavi, R.A. Borzoei, and M.~Aaly Kologani, \emph{Filter theory of pseudo
  hoop-algebras}, Italian Journal of Pure and Applied Mathematics \textbf{37}
  (2017), 619--632.

\bibitem{BlountTsinakis2003}
K.~Blount and C.~Tsinakis, \emph{The structure of residuated lattices},
  Internat. J. Algebra Comput. \textbf{13} (2003), no.~4, 437--461.

\bibitem{BorzoeiNamdarKologani}
R.A. Borzoei, A.~Namdar, and M.~Aaly Kologani, \emph{$n$-fold obstinate filters
  in pseudo-hoop algebras}, Int. J. Industrial Mathematics \textbf{12} (2020),
  147--157.

\bibitem{Botur2011}
M.~Botur, \emph{A non-associative genralization of {H}\'ajek {BL}-algebras},
  Fuzzy Sets and Systems \textbf{178} (2011), 24--37.

\bibitem{BusneagPiciuDina2021}
D.~Bu{\c s}neag, D.~Piciu, and A.~Dina, \emph{Ideals in residuated lattices},
  Carpathian J. Math. \textbf{32} (2021), 53--63.

\bibitem{Castanoetal}
D.~Castano, C.~Cimadamore, J.P.~D\'iaz Varela, and L.~Rueda, \emph{Monadic
  {BL-algebras: {T}he equivalent algebraic semantics of {H}\'ajek's monadic
  fuzzy logic}, journal = {Fuzzy Sets and Systems}, year ={2017}, optkey = {},
  volume = {320}, optnumber = {}, pages = {40--59}, optmonth = {}, optnote =
  {}, optannote = {}}.

\bibitem{ChenDudek2018}
W.~Chen and J.~Dudek, \emph{Ideals and congruences in
  quasi-pseudo-{MV}-algebras}, Soft Computing \textbf{22} (2018), 3879--3889.

\bibitem{ChenXu2023}
W.~Chen and J.~Xu, \emph{Quasi-pseudo-{BL}-algebras and weak filters}, Soft
  Computing \textbf{27} (2023), 2185--2204.

\bibitem{Ciungu}
L.C. Ciungu, \emph{Involutive filters of pseudo-hoops}, Soft Computing
  \textbf{23} (2019), 9459--9476.

\bibitem{EstevaGodo2001}
F.~Esteva and L.~Godo, \emph{Monoidal t-norm based logic: towards a logic for
  left-continuous t-norms}, Fuzzy Sets and Systems \textbf{124} (2001),
  271--288.

\bibitem{Fichtner1970}
K.~Fichtner, \emph{Bemerkung \"uber {M}annigfaltigkeiten universeller
  {A}lgebren mit {I}dealen}, Monatsb. Deutsch. Akad. Wiss. Berlin \textbf{12}
  (1970), 21--25.

\bibitem{GeorgescuLeusteanProtasa}
G.~Georgescu, L.~Leustean, and V.~Preoteasa, \emph{Pseudo-hoops}, Journal of
  Multiple-valued Logic and Soft Computing \textbf{11} (2005), 153--184.

\bibitem{GummUrsini1984}
H.P. Gumm and A.~Ursini, \emph{Ideals in universal algebra}, Algebra
  Universalis \textbf{19} (1984), 45--54.

\bibitem{KondoDudek}
M.~Kondo and W.~Dudek, \emph{Filter theory of {BL}-algebras}, Soft Computing
  \textbf{12} (2008), 419--423.

\bibitem{Mitschke1971}
A.~Mitschke, \emph{Implication algebras are 3-permutable and 3-distributive},
  Algebra Universalis \textbf{1} (1971), 182-- 186.

\bibitem{MotamedMoghaderi2019}
S.~Motamed and J.~Moghaderi, \emph{Some results on filters in residuated
  lattices}, Quasigroup Related Systems \textbf{27} (2019), 91--105.

\bibitem{Oneretal}
T.~Oner, T.~Katican, and A.~Borumand Saeid, \emph{Relation between sheffer
  stroke and hilbert algebras}, Categories and General Algebraic Structures
  with Applications \textbf{14}.

\bibitem{Paad2017}
A.~Paad, \emph{n-fold integral ideals and n-fold {B}oolean ideals in
  {BL}-algebras}, Africa Matematika \textbf{28} (2017), 971--984.

\bibitem{Paad2018}
\bysame, \emph{Folding theory of implicative and obstinate ideals in
  {BL}-algebras}, Discuss. Math. Gen. Alg. and Appl. \textbf{38} (2018),
  255--271.

\bibitem{Paad2021}
\bysame, \emph{Tense operators on {BL}-algebras and their applications}, Bull.
  Sect. Logic \textbf{50} (2021), 299--324.

\bibitem{PaivaRivieccio}
R.~Paiva, R.~Santiago, B.~Bedregal, and U.~Rivieccio, \emph{Inflationary
  {BL}-algebras obtained from 2-dimensional general overlap functions}, Fuzzy
  Sets and Systems \textbf{418} (2021), 64--83.

\bibitem{Panti1999}
G.~Panti, \emph{Varieties of {MV}-algebras}, J. Appl. Non-classical Logic
  \textbf{9} (1999), 141--157.

\bibitem{Ursini1972}
A.~Ursini, \emph{Sulle variet\`a di algebre con una buona teoria degli ideali},
  Boll. Una. Mat. Ital. \textbf{6} (1972), 90--95.

\bibitem{OSV1}
\bysame, \emph{On subtractive varieties {I}}, Algebra Universalis \textbf{31}
  (1994), 204--222.

\bibitem{Vita2014}
M.~Vita, \emph{Why are papers about filters on residuated structures (usually)
  trivial?}, Information Sciences \textbf{276} (2014), 387--391.

\bibitem{VitaCintula2011}
M.~Vita and P.~Cintula, \emph{Filters in algebras of fuzzy logics}, Proceedings
  of the 7th conference of the European Society for Fuzzy Logic and Technology
  (EUSFLAT-11), Advances in Intelligent Systems Research, Atlantis Press, 2011,
  pp.~169--174.

\bibitem{WangHeShe2019}
J.~Wang, P.~He, and Y.~She, \emph{Monadic {NM}-algebras}, Logic Journal of the
  IGPL \textbf{27} (2019), 812--835.

\bibitem{YangSin2017}
RY. Yang and X.~Xin, \emph{On characterizations of {BL}-algebras via
  implicative ideals}, Italian Journal of Pure and Applied Mathematics
  \textbf{37} (2017), 493--506.

\bibitem{LiangZhang2022a}
X.~Zhang and R.~Liang, \emph{Pseudo general overlap functions and weak
  inflationary pseudo {BL}-algebras}, Mathematics \textbf{10} (2022), 3007.

\bibitem{LiangZhang2022b}
X.~Zhang, R.~Liang, and B.~Bedregal, \emph{Weak inflationary {BL}-algebras and
  filters of inflationary (pseudo) general residuated lattices}, Mathematics
  \textbf{10} (2022), 3394.

\end{thebibliography}
\end{document}